\documentclass[11pt,a4paper,reqno]{amsart}

\usepackage[a4paper,lmargin=2cm,rmargin=2cm,tmargin=4cm,bmargin=4cm]{geometry}

\usepackage[english]{babel}

\usepackage[centertags]{amsmath}
\usepackage{amsfonts}
\usepackage{amssymb}
\usepackage{amsthm}
\usepackage{dsfont}

\usepackage[dvipsnames]{xcolor}

\usepackage[normalem]{ulem}
\usepackage{bbm}

\usepackage{hyperref}
	\hypersetup{breaklinks=true,colorlinks=true,
linkcolor=MidnightBlue,citecolor=MidnightBlue,
urlcolor=MidnightBlue}
\usepackage{enumerate}

\usepackage{tikz-cd,graphicx}
\usepackage{tikz}
\usepackage{subcaption}
\usetikzlibrary{shapes.geometric, arrows}
\usetikzlibrary{positioning,arrows,calc}
\tikzstyle{startstop} = [rectangle, rounded corners, minimum width=2cm, minimum height=1cm,text centered,text width=2cm, draw=black]
\tikzstyle{io} = [rectangle, rounded corners, minimum width=2cm, minimum height=1cm,text centered,text width=2cm, draw=black]

\tikzstyle{arrow} = [thick,->,>=stealth]
\tikzstyle{arrow2} = [dashed,->,>=stealth]

\usepackage{dot2texi}

\numberwithin{equation}{section}
\usepackage{enumerate}
\usepackage{mathtools}

\usepackage[shortlabels]{enumitem}

\newcommand{\abs}[1]{\left\lvert#1\right\rvert}
\newcommand{\norm}[1]{\left\lVert#1\right\rVert}

\newcommand{\N}{\mathbb{N}}
\newcommand{\M}{\mathbb{M}}

\newcommand{\A}{\mathcal{A}}
\newcommand{\B}{\mathcal{B}}

\newcommand{\eps}{\varepsilon}

\newcommand{\scal}[1]{\left\langle#1\right\rangle}

\newcommand{\nnorm}[1]{{\left\vert\kern-0.25ex\left\vert\kern-0.25ex\left\vert #1%
    \right\vert\kern-0.25ex\right\vert\kern-0.25ex\right\vert}}

\DeclareMathOperator{\dent}{dent}
\DeclareMathOperator{\Id}{Id}
\DeclareMathOperator{\diam}{diam}

\DeclareMathOperator{\supp}{supp}
\DeclareMathOperator{\conv}{conv}

\DeclareMathOperator{\cconv}{\overline{conv}}

\DeclareMathOperator{\SCD}{SCD}

\DeclareMathOperator{\sign}{sign}
\DeclareMathOperator{\spn}{span}

\DeclareMathOperator{\ext}{ext}
\DeclareMathOperator{\strexp}{str-exp}

\DeclareMathOperator{\PC}{PC}

\newcommand{\T}{\mathbb{T}}

\renewcommand{\geq}{\geqslant}
\renewcommand{\leq}{\leqslant}

\theoremstyle{plain}
\newtheorem{theorem}{Theorem}[section]
\newtheorem{lemma}[theorem]{Lemma}
\newtheorem{proposition}[theorem]{Proposition}
\newtheorem{corollary}[theorem]{Corollary}

\theoremstyle{definition}

\newtheorem*{definition*}{Definition}
\newtheorem{remark}[theorem]{Remark}

\newtheorem{question}[theorem]{Question}

\begin{document}

\title{A Banach space with an unconditional basis which is not slicely countably determined}

\author[L\~{o}o]{Marcus L\~{o}o}
\address[L\~{o}o]{Institute of Mathematics and Statistics, University of Tartu, Narva mnt 18, 51009 Tartu, Estonia}
\email{marcus.loo@ut.ee}
\urladdr{
\href{https://orcid.org/0009-0003-1306-5639}{ORCID: \texttt{0009-0003-1306-5639} } }

\author[Perreau]{Yo\"el Perreau}
\address[Perreau]{Institute of Mathematics and Statistics, University of Tartu, Narva mnt 18, 51009 Tartu, Estonia}
\email{yoel.perreau@ut.ee}
\urladdr{
\href{https://orcid.org/0000-0002-2609-5509}{ORCID: \texttt{0000-0002-2609-5509} } }




\begin{abstract}
In this note, we study the geometry of the unit ball of the Banach space generated by the adequate family of all subsets of branches of the infinite binary tree, and answer several open questions related to slicely countably determined Banach spaces. Our main result is that the binary tree space is an example of a Banach space with an unconditional basis which fails to be slicely countably determined. In particular, it provides an example of a non slicely countably determined separable Banach space which contains no isomorphic copy of a space with the Daugavet property. We also exhibit some other geometric features of this space: we prove that its unit ball is dentable, that it has numerical index~1, and that the points of continuity of its unit ball form a weakly dense set.  Finally, we show that the binary tree space contains a non-convex subset which is slicely countably determined, but does not admit a countable $\pi$-base for its relative weak topology, and that there is a 2-equivalent renorming of this space whose unit ball fails to be slicely countably determined. 
\end{abstract}

\maketitle

\section{Introduction}

Slicely countably determined Banach spaces (SCD spaces for short) were introduced in \cite{AKMMS10} to provide a class of separable Banach spaces that would contain both spaces with the Radon--Nikod\'ym property (RNP for short) and spaces not containing an isomorphic copy of the space $\ell_1$, as well as other natural spaces constructed out of these such as $\ell_1(c_0)$ or $c_0(\ell_1)$ (see Section~\ref{section:preliminaries} for definitions). Incidentally, the corresponding property for sets and operators also provided a powerful tool for the study of some classical properties of Banach spaces such as the Daugavet and alternative Daugavet properties, numerical index~1, or spear operators. Given a (real or complex) Banach space $X$, we denote its topological dual as $X^*$, and its unit ball and unit sphere as $B_X$ and $S_X$, respectively. Recall that $X$ has the \emph{Daugavet property} if every rank 1 operator $T:X\to X$ satisfies the equation \begin{equation*}
    \norm{\Id+T}=1+\norm{T},
\end{equation*}
and the \emph{alternative Daugavet property} if every rank 1 operator $T:X\to X$ satisfies the equation
\begin{equation*}
    \max_{\theta\in \T}\norm{\Id+\theta T}=1+\norm{T},
\end{equation*} 
where $\Id$ denotes the identity operator of $X$ and $\mathbb{T}$ the set of modulus one scalars. Also, the Banach space $X$ is said to have \emph{numerical index~1} if every bounded linear operator $T:X\to X$ has numerical radius~1, i.e. satisfies the equation \begin{equation*}
    \norm{T}=\sup\{\abs{f(Tx)}\colon x\in S_X,\ f\in S_{X^*},\ f(x)=1\}.
\end{equation*}
Spaces with numerical index~1 satisfy the alternative Daugavet property, but the converse is known to be false in general. However, when the unit ball $B_X$ of the considered space $X$ is an SCD set, then the latter two properties coincide (and in fact coincide with a stronger geometric property called lushness, see \cite{AKMMS10} for more details). We refer to the monographs \cite{KMMP,KMRZW25} for a detailed study of these properties and their applications to the study of bounded linear operators and the geometry of Banach spaces. 

So far, the only known property preventing a separable Banach space from being SCD is the Daugavet property. In particular, the following questions were raised in \cite{AKMMS10} and subsequently in \cite{KPW18}.

\begin{question}\label{question:non_SCD_implies_isomorphic_copy_of_Daugavet_space}
    Let $X$ be a separable Banach space which fails to be SCD. Does $X$ have to contain an isomorphic copy of a Banach space with the Daugavet property?
\end{question}
\begin{question}\label{question:unconditional_basis_implies_SCD}
    Let $X$ be a Banach space with an unconditional basis. Is $X$ necessarily an SCD space? 
\end{question}
It is a well known fact that spaces with the Daugavet property do not embed isomorphically into Banach spaces with an unconditional basis, so a negative answer to Question \ref{question:unconditional_basis_implies_SCD} would also answer Question \ref{question:non_SCD_implies_isomorphic_copy_of_Daugavet_space}. The authors of \cite{AKMMS10} provided many examples for which the answer to the latter question was positive (e.g. shrinking and boundedly complete unconditional sums of SCD spaces), and it was later proved in \cite{KMMW2013} that if $X$ is a Banach space with a 1-unconditional basis, then the unit ball $B_X$ of $X$ is an SCD set. In particular, numerical index~1 and the alternative Daugavet property coincide in this context. Nevertheless, we will provide a negative answer to Question~\ref{question:unconditional_basis_implies_SCD}, hence also to Question~\ref{question:non_SCD_implies_isomorphic_copy_of_Daugavet_space}, by producing an example of a Banach space with a 1-unconditional basis whose positive unit ball is not an SCD set.

Our target will be the so-called binary tree space $X_T$, which can be seen as the Banach space generated by the adequate family of subsets of branches of an infinite binary tree $T$ in the sense of Talagrand (see \cite{Tal79} and \cite{Tal84}). There are several reasons that make this space particularly relevant to Question~\ref{question:unconditional_basis_implies_SCD}. First and foremost, this space is connected dually to the stopping time Banach space, introduced by Rosenthal in an unpublished manuscript, and which was shown by Bang and Odell in \cite{BO89} to be universal for spaces with a 1-unconditional basis and isometric embeddings. In particular, the latter implies that the existence of a Banach space with an unconditional basis which would fail to be SCD is actually equivalent to the fact that $X_T$ fails to be SCD. Second, a modified version $X_\mathfrak{M}$ of the binary tree was shown in \cite{ALMT} to have some striking geometric properties, namely it has a 1-unconditional basis and contains a weakly dense subset of Daugavet points (see \cite{AHLP} for the definition of Daugavet points), pointing towards a potential non-SCDness in the space. Our main result is the following.

\begin{theorem}\label{intro_theorem:the_binary_tree_space_is_not_SCD}
    The positive unit ball $B_{X_T}^+$ of the binary tree space is not an SCD set. In particular, $X_T$ is a space with an unconditional basis which fails to be SCD.
\end{theorem}

We will also show that the symmetrization of the set $B_{X_T}^+$ fails to be SCD, which in particular yields the following.

\begin{theorem}\label{intro_theorem:a_Banach_space_with_2-unconditional_basis_whose_unit_ball_is_not_SCD}
    There exists a Banach space $Y\simeq X_T$ with a 2-unconditional basis such that $B_Y$ is not an SCD set.
\end{theorem}

Last, we obtain a few results related to the notion of countable $\pi$-bases (see Section~\ref{section:preliminaries} for the definition). As observed in \cite{AKMMS10}, if a convex bounded non-empty subset $A$ of a Banach space $X$ has a countable $\pi$-base for the relative weak topology, then the set $A$ is SCD. On the other hand, if $A$ is separable and if $(A,w)$ fails to have a countable $\pi$-base, then $A$ must contain a sequence equivalent to the canonical basis of $\ell_1$. It is an open question, asked in \cite{AKMMS10} and \cite{KPW18}, whether every convex bounded set which is SCD also has a countable $\pi$-base for the relative weak topology. To the best of the authors' knowledge, it is also unknown whether or not the unit ball of a space with a 1-unconditional basis has a countable $\pi$-base for its relative weak topology. We will show that the strengthening of the alternative Daugavet property introduced in \cite{LLMRZP25} naturally prevents the existence of a countable $\pi$-base for the relative weak topology of the unit ball of the considered space. In particular, all these spaces have to contain an isomorphic copy of $\ell_1$, which was left as an open question in \cite{LLMRZP25}. 

\begin{theorem}\label{intro_theorem:super_ADP_and_countable_pi-base}
    Let $X$ be a Banach space with the super alternative Daugavet property. Then $(B_X,w)$ fails to have a countable $\pi$-base. In particular, $X$ contains a copy of $\ell_1$. 
\end{theorem}

Using similar ideas, we will also exhibit a non-convex subset $\Sigma$ of $B_{X_T}$ which is SCD, but has no countable $\pi$-base for the relative weak topology. On the other hand, we will show that $B_{X_T}=\cconv(\Sigma)$ has a countable $\pi$-base for its relative weak topology. We show analogous results for the modified binary tree space and the countably branching tree space $X_{T_\infty}$, obtaining the following. 

\begin{theorem}\label{intro_theorem:tree_spaces_unit_balls_have_countable_pi-bases}
    Let $X=X_T,X_\mathfrak{M}$ or $X_{T_\infty}$. Then $(B_X,w)$ has a countable $\pi$-base. 
\end{theorem}

Let us end this section with a quick description of the structure of the paper. In Section~\ref{section:preliminaries}, we introduce all the necessary definitions, and provide a short proof of the fact that the unit ball of a space with the Daugavet property fails to be SCD. Then we show how a similar scheme leads to Theorem~\ref{intro_theorem:super_ADP_and_countable_pi-base}. Our main results are presented in Section~\ref{section:binary_tree_space}, where we conduct a detailed study of the geometric properties of the unit ball of the binary tree space. In Subsection~\ref{subsec:binary_tree_dentability}, we show that $B_{X_T}$ is dentable, and more precisely, that it is the closed convex hull of its strongly exposed points. As a consequence, we obtain that $X_T$ has the alternative Daugavet property, hence numerical index~1. In Subsection~\ref{subsec:binary_tree_not_SCD}, we prove Theorem~\ref{intro_theorem:the_binary_tree_space_is_not_SCD} by showing that the positive unit ball of the space $X_T$ behaves geometrically like a ``Daugavet set'', and exhibit a non-convex SCD subset of $B_{X_T}$ which has no countable $\pi$-base for the relative weak topology. In Subsection~\ref{subsec:binary_tree_renorming}, we recall the definition of SCD points, and study these in $B_{X_T}^+$ and its symmetrization $C:=\cconv(B_{X_T}^+\cup B_{X_T}^-)$ to show that the latter set also fails to be SCD. In particular, we obtain a 2-unconditional equivalent renorming of $X_T$ whose unit ball is not SCD, yielding Theorem~\ref{intro_theorem:a_Banach_space_with_2-unconditional_basis_whose_unit_ball_is_not_SCD}. Finally, in Subsection~\ref{subsec:weak_density_of_points_of_continuity}, we prove that the points of continuity of $B_{X_T}$ and of $B_{X_\mathfrak{M}}$ are weakly dense, and infer that the relative weak topologies of these two sets admit a countable $\pi$-base. We extend the latter result to the countably branching tree space in Section~\ref{section:countably_branching_tree_space}. The unit ball of this space has no points of continuity, but we obtain nonetheless the existence of a $\pi$-base by using a reduction argument which shows that every functional on $X_{T_\infty}$ is essentially supported on a finitely branching subtree. We end the paper with some open questions and a few additional miscellaneous results related to the topic in Section~\ref{section:open_questions}.

\section{Preliminaries}\label{section:preliminaries}

We will use standard Banach space notation following \cite{FHHMZ11}. For simplicity, we will consider real Banach spaces only. Let $X$ be a Banach space and $A$ a bounded non-empty subset of $X$. We denote respectively by $\conv(A)$ and $\cconv(A)$ the convex hull of $A$ and its closure, by $\spn(A)$ and $[A]$ the linear span of $A$ and its closure. A \emph{slice} of $A$ is a non-empty intersection of $A$ with an open half-space, or in other words a subset of $A$ of the form \begin{equation*}
    S(A,f,\delta):=\{x\in A\colon f(x)>\sup f(A)-\delta\},
\end{equation*} where $f\in X^*$ and $\delta>0$. We say that $x\in A$ is
\begin{enumerate}
    \item an \emph{extreme point} of $A$ (writing $x\in \ext(A)$) if $x$ does not belong to the interior of any segment of $A$;
    \item a \emph{point of continuity} of $A$ (writing $x\in \PC(A)$) if the identity mapping $\Id\colon (A,w)\to(A,\norm{\cdot})$ is continuous at $x$ (that is, if $x$ is contained in relatively weakly open subsets of $A$ of arbitrarily small diameter);
    \item a \emph{denting point} of $A$ (writing $x\in\dent(A)$) if $x$ is contained in slices of $A$ of arbitrarily small diameter;
    \item a \emph{strongly exposed point} of $A$ (writing $x\in \strexp(A)$) if there exists $f\in X^*$ such that for all sequences $(x_n)\subseteq A$,  $f(x_n)\rightarrow f(x)$ if and only if $x_n\rightarrow x$ in norm.
\end{enumerate}

 We say that a subset $G$ of $A$ (respectively a family $(x_i)_{i\in I}$ of elements of $A$) is \emph{generating} for $A$ if $A\subseteq\cconv(G)$ (respectively $A\subseteq \cconv\{x_i\colon i\in I\}$). By the Hahn--Banach theorem, $G$ is generating for $A$ if and only if every slice of $A$ contains an element of $G$. In particular, a subset $G$ of $B_X$ is generating for $B_X$ if and only if it is norming for $X^*$. We say that a collection $\{A_i\}_{i\in I}$ of non-empty subsets of $A$ is \emph{determining} for $A$ if every $(x_i)_{i\in I}\in \prod_{i\in I}A_i$ is generating for $A$. Equivalently, the collection $\{A_i\}_{i\in I}$ is determining for $A$ if and only if every slice of $A$ contains one of the sets $A_i$. The set $A$ is said to be \emph{slicely countably determined} (SCD set for short) if it admits a countable determining collection of slices.

In \cite[Remark~2.7]{AKMMS10} \cite[Proposition~7.20]{KMMP} it was observed that a set $A$ is SCD if and only if $\cconv(A)$ is SCD. Furthermore, it follows from Bourgain lemma (see \cite[Lemma~II.1]{GGMS}) that if $A$ is assumed to be convex, then SCDness is equivalent to the existence of a countable determining collection of relatively weakly open subsets. Prototypical examples of SCD sets are separable dentable sets, i.e. sets $A$ for which $\dent(A)$ is generating, and separable sets not containing a sequence equivalent to the canonical basis of $\ell_1$. This notion was first introduced in \cite{AKMMS10} for convex sets, and was later extended to non-convex sets in \cite{KPW18}. A separable Banach space $X$ is said to be \emph{slicely countably determined} (SCD space for short) if every bounded non-empty subset of $X$ is SCD. Considering the examples of SCD sets mentioned before, we have that separable Banach spaces with the RNP (spaces where every bounded convex and closed subset is dentable) and separable spaces not containing an isomorphic copy $\ell_1$ (in particular, spaces with a separable dual) are SCD. A thorough study of SCD sets and spaces is presented e.g. in \cite[Chapter~10]{KMRZW25}.

Concerning the Daugavet property and the alternative Daugavet property, let us recall the following well known geometric characterizations via slices of the unit ball. 
\begin{proposition}[\cite{KSSW00}, \cite{MT04}]\label{prop:DP_and_ADP_slice_characterization}
    Let $X$ be a Banach space. Then
    \begin{itemize}
        \item $X$ has the Daugavet property if and only if $\underset{y\in S}{\sup}\;||x+y||=2$ for every $x\in S_X$ and every slice $S$ of $B_X$;
        \item $X$ has the alternative Daugavet property if and only if $\underset{y\in S}{\sup}\;\underset{\theta\in \mathbb{T}}{\max}\;||x+\theta y||=2$ for every $x\in S_X$ and every slice $S$ of $B_X$.
    \end{itemize}
\end{proposition}

It is known that spaces with the Daugavet property are never SCD (see \cite[Example~2.13]{AKMMS10}). Using the characterization above, we provide a short sketch proof below for the sake of completeness.

\begin{theorem}\label{thm:Daugavet_unit_balls_are_not_SCD}
    Let $X$ be a Banach space with the Daugavet property. Then $B_X$ fails to be SCD. 
\end{theorem}

\begin{proof}
Since $X$ has the Daugavet property, then given any countable collection $\{S_n\}_{n\in \N}$ of slices of $B_X$ and $\eps\in(0,1)$, we can find $(x_n)\in \prod_{n\in\N} S_n$ such that for every $n\in\N$, $\norm{x_1+\dots+x_n}>n-\eps$.
   
More precisely, let $(\eps_n)$ be a sequence of positive real numbers such that $\sum_{n\in\N}\eps_n<\eps$. Then we can construct inductively the sequence $(x_n)\in\prod_{n\in\N} S_n$ in such a way that for every $n\in\N$, \begin{equation*}
\norm{x_1+\dots+x_n}>n-\sum_{i=1}^n\eps_i.
\end{equation*} Indeed, we can first pick an arbitrary $x_1\in S_1$, and assuming that $x_1,\dots,x_n$ have been constructed in such a way, then we can just apply the Daugavet property to the point $x=\frac{x_1+\dots+x_n}{\norm{x_1+\dots+x_n}}$, the slice $S_{n+1}$ and $\xi=\frac{\eps}{\norm{x_1+\dots+x_n}}$ to get the point $x_{n+1}$ we were looking for. 

Now this means that for every $n\in \N$, we can find $f_n\in S_{X^*}$ such that for every $i\in\{1,\dots,n\}$, we have $f_n(x_i)>1-\eps$. So if $f\in B_{X^*}$ is a cluster point of the sequence $(f_n)$ for the weak$^*$ topology, then we have that every $x_n$ belongs to the (closed and convex) set $\{x\in B_X\colon f(x)\geq 1-\eps\}$. In particular, $0\notin \cconv\{x_n\colon n\in\N\}$, and $\{S_n\}_{n\in \N}$ is not determining for $B_X$. 
\end{proof}

\begin{remark}\label{rem:Daugavet_sets_are_not_SCD}
Let us note that in the above proof, we would get the exact same conclusion if we were to replace the unit ball of the space $X$ by some non-empty closed convex subset $A$ which contains $0$ and satisfies the following condition: for every $\eps>0$, $a\in A\cap S_{X}$, and $S$ slice of $A$, there exists $b\in S$ such that $\norm{a+b}\geq 2-\eps$. Surprisingly, we will show in the next section that such ``Daugavet sets'' may be found in the unit ball of Banach spaces with a 1-unconditional basis.
\end{remark}

Let $(X,\tau)$ be a topological space. Recall that a collection $\{V_i\}_{i\in I}\subseteq\tau$ of non-empty open subsets of $X$ is a called a \emph{$\pi$-base} for $(X,\tau)$ if every non-empty open subset of $X$ contains one of the open sets $V_i$. Equivalently, the collection $\{V_i\}_{i\in I}$ is a $\pi$-base for $(X,\tau)$ if and only if every $(x_i)_{i\in I}\in\prod_{i\in I} V_i$ is $\tau$-dense in $X$, i.e. is such that $X=\overline{\{x_i\colon i\in I\}}^\tau$.

In particular, let $X$ be a Banach space and $A$ be a bounded non-empty subset of $X$. If $(A,w)$ has a countable $\pi$-base, then $A$ admits a countable determining collection of non-empty relatively weakly open subsets.  As previously mentioned, if $A$ is additionally assumed to be convex, then the latter property implies that $A$ is SCD. On the other hand, let us point out that the example from \cite[Proposition~10.7.6]{KMRZW25} actually also shows that for non-convex sets, the existence of a countable $\pi$-base does not always implies SCDness. Prototypical examples of sets which admit a countable $\pi$-base for the relative weak topology are separable sets in which points of continuity are weakly dense, and separable sets not containing a sequence equivalent to the canonical basis of $\ell_1$. The latter is a consequence of some deep topological studies from Todorčević (see the proof of \cite[Theorem~2.22]{AKMMS10} for more details). To the best of the authors' knowledge, this is actually the only known way of proving that a separable set which fails to be SCD has to contain such a sequence. Let us also point out that it was proved in \cite[Section~6.2]{AKMMS10} that every non-empty bounded convex subset of $c_0(\ell_1)$ or $\ell_1(c_0)$ has a countable $\pi$-base for the relative weak topology. 

From the above we know that the unit ball of a Banach space with the Daugavet property fails to have a countable $\pi$-base for the relative weak topology. It turns out that we can extend this result to all spaces with the super alternative Daugavet property, introduced in \cite{LLMRZP25}. Recall that an infinite-dimensional Banach space $X$ has the \emph{super alternative Daugavet property} (super ADP for short) if for every $x\in S_X$ and every non-empty relatively weakly open subset $W$ of $B_X$ we have
\[
\underset{y\in S}{\sup}\;\underset{\theta\in \mathbb{T}}{\max}\;||x+\theta y||=2.
\]
In \cite{LLMRZP25} it is proved that the super ADP lies strictly between the Daugavet property and the alternative Daugavet property. The following theorem shows, in particular, that a Banach space with the super ADP contains an isomorphic copy of $\ell_1$, solving an open question from \cite{LLMRZP25}. Let us point out that it is still currently open whether or not the unit ball of such a space can be SCD. 
\begin{theorem}\label{thm:super_ADP_unit_balls_do_not_have_a_weak_pi_base}
    Let $X$ be a Banach space with the super ADP. Then $(B_X,w)$ does not have a countable $\pi$-base. In particular, $X$ contains an isomorphic copy of $\ell_1$.
\end{theorem}
\begin{proof}
As above, it is easy to prove that since $X$ has the super ADP, then given any countable collection $\{V_n\}_{n\in\N}$ of non-empty relatively weakly open subsets of $B_X$ and $\eps\in(0,1)$, we can find a sequence $(v_n)\in\prod_{n\in\N}V_n$ and a sequence $(\theta_n)\subseteq\{-1,1\}$ such that for every $n\in\N$, $\norm{\theta_1 v_1+\dots+\theta_n v_n}> n-\eps$. In particular, we can find $f_n\in S_{X^*}$ such that for every $i\in\{1,\dots,n\}$, $f_n(\theta_iv_i)> 1-\eps$.  So if $f\in B_{X^*}$ is a cluster point of the sequence $(f_n)$ for the weak$^*$ topology, then we have that every $v_n$ belongs to the (weakly closed) set $\{x\in B_X\colon \abs{f(x)}\geq 1-\eps\}$. In particular, $0\notin\overline{\{v_n\colon n\in \N\}}^w$, and $\{V_n\}_{n\in\N}$ is not a $\pi$-base for $(B_X,w)$. 
    
For the second part of the statement, recall that it was proved in \cite{LLMRZP25} that there is a separable subspace $Y$ of $X$ (in fact an almost isometric ideal) which also has the super ADP. By the above, $(B_Y,w)$ fails to have a countable $\pi$-base, and it follows that $Y$ has to contain a copy of $\ell_1$ (see \cite[Theorem~10.1.21]{KMRZW25}).  
\end{proof}

\section{The binary tree space}\label{section:binary_tree_space}

Recall that a \emph{tree} is a preordered set $T:=(T,\preceq)$ with the property that all the sets of the form $\{s\preceq t\colon s\in T\}$ with $t\in T$ are well ordered. We say that an element $t\in T$ is a \emph{successor} of $s\in T$ if $s\prec t$ (that is, $s\preceq t$ and $s\neq t$), and we say that $t$ is an \emph{immediate successor} of $s$ if $s\prec t$ and if there is no $r\in T$ such that $s\prec r \prec t$. A \emph{chain} of $T$ is a well ordered non-empty subset of $T$, and an \emph{antichain} of $T$ is a non-empty subset of $T$ composed of pairwise incomparable elements for $\preceq$. We also call \emph{branch} of $T$ any maximal chain in $T$. 

In this section, $T:=\{\emptyset\}\cup\bigcup_{n\in\N}\{0,1\}^n$ will be the infinite binary tree with its natural preorder $\preceq$. For every $t\in T$, we denote by $\abs{t}$ the length of the sequence $t$ (with the convention that $\abs{\emptyset}=0$) and by $t_{\lvert k}$ the sequence obtained by taking the first $k$ elements of $t$ ($k\leq \abs{t}$). Recall that $s\preceq t$ if and only if $s=t_{\lvert k}$ for some $k\leq\abs{t}$. Given $s,t\in T$, we denote by $s\smallfrown t$ the concatenation of the sequences $s$ and $t$, i.e. the sequence of length $\abs{s\smallfrown t}=\abs{s}+\abs{t}$ obtained by taking first the $\abs{s}$ elements of $s$ and then the $\abs{t}$ elements of $t$. We usually write $n:=(n)$ the length-one elements of $T$ and $s\smallfrown n:=s\smallfrown (n)$. 

Let $\A$ be the collection of all chains in $T$. We consider the real Banach space $X:=X_T$ obtained by taking the completion of the space $c_{00}(T)$ of all finitely supported functions on $T$ equipped with the norm given by \begin{equation*}
\norm{x}:=\sup_{A\in \A}\sum_{t\in A}\abs{x(t)}
\end{equation*} for every $x\in c_{00}(T)$. Observe that the collection $\A$ forms an adequate family of subsets of $T$ (see Section~\ref{section:open_questions} for more details). In particular, for every $x\in X$, we can find a maximal subset $A$ of $\A$ (that is a branch of $T$) such that $\norm{x}=\sum_{s\in A}\abs{x_s}$. Furthermore, if $e_t$ denotes the function on $T$ taking value 1 at $t$ and $0$ elsewhere, then the family $(e_t)_{t\in T}$ is a 1-unconditional basis for $X$. Also note that for every branch $\beta$ of $T$, we have that the sequence $(e_t)_{t\in\beta}$ is isometrically equivalent to the canonical basis of $\ell_1$, and that for every infinite antichain $\alpha$ of $T$, we have that the sequence $(e_t)_{t\in\alpha}$ is isometrically equivalent to the canonical basis of $c_0$. 

Throughout the text, we will use the following notation. Given any subset $A$ of $X$, we denote by $A^+$ the intersection of $A$ with the positive cone $X^+$ of $X$, and by $A^-$ the intersection of $A$ with the negative cone $X^-$ of $X$. For every $n\in\N$, we set \begin{equation*}
    L_n:=\{s\in T\colon \abs{s}=n\}, \quad T_n:=\{s\in T\colon \abs{s}\leq n\}, \quad \text{ and } \quad \perp_n:=T\setminus T_n=\{s\in T\colon \abs{s}>n\}.
\end{equation*} Also, for every $s\in T$, we set \begin{equation*}
    T(s):=\{t\in T\colon t\succeq s\}.
\end{equation*} We denote by $\A_f$ the collection of all finite chains in $T$. Given $A\in\A_f$ and $S\subseteq T$, we write \begin{equation*}
    A\prec S
\end{equation*} if $s\succ \max_{\preceq}(A)$ for every $s\in S$, or in other words if $S\subseteq T(\max_{\preceq}(A))\setminus \{\max_{\preceq}(A)\}$. 

Given a non-empty subset $S$ of $T$, we denote by $P_S$ the projection on the subspace $[e_s]_{s\in S}$. The following lemma is straightforward, and we will use it repeatedly without mention in the rest of the text. 

\begin{lemma}\label{lem:totally_incomparable_decomposition}
    Let $\Gamma$ be a non-empty set, $\{S_\gamma\}_{\gamma\in\Gamma}$ be a family of pairwise totally incomparable subsets of $T$ for $\preceq$, and $S:=\bigcup_{\gamma\in\Gamma}S_\gamma$. Then the set $\Gamma$ is countable, and we have the $M$-decomposition \begin{equation*}
        [e_s]_{s\in S}=\bigg(\sum_{\gamma\in\Gamma}[e_s]_{s\in S_\gamma}\bigg)_{c_0}.
    \end{equation*}
\end{lemma}

\subsection{The unit ball of the binary tree space is dentable}\label{subsec:binary_tree_dentability} Since the space $X$ has a 1-unconditional basis, then it follows from \cite[Theorem~3.1]{KMMW2013} that $B_X$ is an SCD set. In this subsection, we will prove that the strongly exposed points of $B_X$ are generating, which implies that $B_X$ is even dentable. As an application, we will see that the space $X$ also has the alternative Daugavet property, hence numerical index 1. 

\begin{proposition}\label{prop:stongly_exposed_points_characterization}
    Let $x\in S_X$. Then $x$ is a strongly exposed point of $B_X$ if and only if there exists a finite maximal antichain $\alpha$ of $T$ such that $\abs{x(s)}=1$ for every $s\in\alpha$. In particular, $\ext(B_X)=\strexp(B_X)$.
\end{proposition}

 \begin{proof}
 Clearly, a point $x\in S_X$ is an extreme point of $B_X$ if and only if the above condition is satisfied. Furthermore, assuming that this condition is satisfied, then we have that every branch of $T$ intersects the support of $x$, and therefore the norm one functional \begin{equation*}
        f = \frac{1}{\abs{\alpha}}\sum_{s\in \alpha}\sign(x(s))e^*_s
\end{equation*} is strongly exposing at $x$. 
\end{proof}

Let $\Pi$ be the set of all strongly exposed points of $B_X$, and for every $n\in\N$, let \begin{equation*}
    X_n:=\spn\{e_t\}_{t\in T_n} \quad\text{and}\quad  \Pi_n:=\Pi\cap X_n.
\end{equation*} 

\begin{lemma}
   The set $\Pi_n$ is generating for $B_{X_n}$.
\end{lemma}

\begin{proof}
    Again, it is straightforward to check that $\Pi_n=\ext(B_{X_n})$. Since $\dim(X_n)<\infty$, the conclusion immediately follows from the Minkowski theorem. 
\end{proof}

\begin{corollary}\label{cor:strongly_exposed_points_are_generating_for_B_X}
  The set $\Pi$ is generating for $B_X$. In particular, $B_X$ is dentable, hence SCD.
\end{corollary}

\begin{proof}
    This immediately follows from the previous result and the fact that $X=\overline{\bigcup_{n\in\N}X_n}$. The in particular part holds because $\Pi\subseteq \dent(B_X)$.
\end{proof}

With the latter result at hand, we can now easily show that the space $X$ enjoys the alternative Daugavet property, hence numerical index 1.

\begin{theorem}
    The space $X$ has the alternative Daugavet property. In particular, it has numerical index 1.
\end{theorem}

\begin{proof}
    We will make use of Proposition \ref{prop:DP_and_ADP_slice_characterization}. Let $x\in S_X$ and $S$ be a slice of $B_X$. By the above result, $S$ contains a strongly exposed point $y$ of $B_X$. By Proposition~\ref{prop:stongly_exposed_points_characterization}, we can find a finite maximal antichain $\alpha$ such that $\abs{y(s)}=1$ for every $s\in\alpha$. Let $\beta$ be a branch of $T$ such that $\norm{x}=\sum_{s\in\beta}\abs{x(s)}$. Since $\alpha$ is maximal, there exists (a unique) $t\in T$ which belongs to both $\alpha$ and $\beta$. Let $\theta:=\frac{\sign(x(t))}{\sign(y(t))}$. Then $\abs{x(t)+\theta y(t)}=1+\abs{x(t)}$, and therefore \begin{equation*}
        \norm{x+\theta y}\geq \sum_{s\in\beta}\abs{x(s)+\theta y(s)}= 1+\sum_{s\in\beta}\abs{x(s)}=2.
    \end{equation*}The conclusion follows. 

For the second part of the statement, recall that it was proved in \cite[Theorem~4.4]{AKMMS10} that every Banach space with the alternative Daugavet property whose unit ball is an SCD set satisfies a stronger geometric condition called lushness, and that this property is known to imply numerical index 1. 
\end{proof}

\subsection{The binary tree space is not an SCD space}\label{subsec:binary_tree_not_SCD} In the last subsection, we saw that the unit ball of $X$ is an SCD set. In turns out however that its positive unit ball is not, which means that $X$ is not an SCD space. 

\begin{theorem}\label{thm:binary_tree_space_is_not_SCD}
The positive unit ball $B_X^+$ of $X$ is not an SCD set. In particular, $X$ is a Banach space which fails to be SCD, but contains no isomorphic copy of a Banach space with the Daugavet property. 
\end{theorem}

\begin{proof}
We will prove that for every $x\in S_X^+$ and every slice $S$ of $B_X^+$, we can find an element $y\in S$ such that $\norm{x+y}=2$. In particular, $B_X^+$ is a ``Daugavet set'' in the sense of Remark~\ref{rem:Daugavet_sets_are_not_SCD}, and thus it is not an SCD set. So take $x$ and $S$ in this way, and fix a finitely supported element $y\in S$. Wlog, $x$ is also finitely supported. Let $A\in\A_f$ be such that $\sum_{t\in A}x(t)=1$. Additionally, let $B\in \A_f$ be such that $B\succ A$ and $y(s)=0$ for every $s\succ \max_\preceq(B)$. Then, select an infinite antichain $\alpha:=(\alpha(n)) \succ B$ of $T$, and for every $n\in \N$, let \begin{equation*}
    y_n:=y+(1-\lambda)e_{\alpha(n)},
\end{equation*} where $\lambda:=\sum_{t\in A\cup B}y(t)$. By construction, $y_n\in B_X^+$, and since $(e_{\alpha(n)})$ is weakly null (it is equivalent to the canonical basis of $c_0$),
the sequence $(y_n)$ converges weakly to $y$. So there exists $n_0\in\N$ such that $y_{n_0}$ belongs to $S$. Now we have \begin{equation*}
    \norm{x+y_{n_0}}\geq \sum_{t\in A\cup B\cup \{\alpha(n_0)\}}\abs{x(s)+y_{n_0}(s)}=\sum_{t\in A}x(t)+\sum_{t\in A\cup B}y(t)+y(\alpha(n_0))=1+\lambda+(1-\lambda)=2. 
\end{equation*} The conclusion follows.

The second part of the statement follows from the fact that a Banach space with the Daugavet property cannot embed linearly into a Banach space with an unconditional basis (see \cite[Corollary~2.7]{KSSW00}). 
\end{proof}

Using similar ideas, we can produce a (non-convex) subset of $B_X$ which is SCD, but has no countable $\pi$-base for the relative weak topology. To the best of the authors' knowledge, this is the first known example of such a set.

\begin{proposition}\label{prop:non-convex_SCD_set_which_is_not_WCS}
    Consider the subset $\Sigma$ of $B_X$ given by \begin{equation*}
        \Sigma:=\{x\in B_X\colon \forall s\in T,\ x(s)\in \{-1,0,1\}\}.
    \end{equation*} Then $\Sigma$ is an SCD set, but $(\Sigma,w)$ has no countable $\pi$-base.
\end{proposition}

\begin{proof}
As already mentioned, to show that the set $\Sigma$ is SCD, it is sufficient (in fact equivalent) to prove that its closed convex hull is an SCD set. Since $\Sigma$ contains all the strongly exposed points of $B_X$, we get from Corollary~\ref{cor:strongly_exposed_points_are_generating_for_B_X} that $\cconv(\Sigma)=B_X$, which is indeed SCD.

To show that $\Sigma$ has no countable $\pi$-base for the relative weak topology, we will establish that given any sequence $(V_n)$ of non-empty relatively weakly open subsets of $\Sigma$, we can find $(x_n)\in \prod_{n\in \N}V_n$ and $(\theta_n)\subseteq \{-1,1\}$ such that for every $n\in\N$, we have \begin{equation*}
    \norm{\theta_1x_1+\dots+\theta_nx_n}=n.
\end{equation*} The conclusion then follows as in the proof of Theorem~\ref{thm:super_ADP_unit_balls_do_not_have_a_weak_pi_base}. So fix such a sequence $(V_n)$. We will construct the sequence $(x_n)$ inductively. For the first step, simply note that 0 is not an isolated point of $\Sigma$, because given any infinite antichain $\alpha:=(\alpha(n))$ of $T$, we have that the sequence $(e_{\alpha(n)})$ converges weakly to 0. So we can pick any non-zero $x_1\in V_1$ and $\theta_1\in\{-1,1\}$. Now assume that we have constructed $x_1,\dots,x_n$ and $\theta_1,\dots,\theta_n$ such that $\norm{\theta_1x_1+\dots+\theta_nx_n}=n$. Then we can find $A\in\A_f$ such that \begin{equation*}
    \sum_{t\in A}\abs{\theta_1x_1(t)+\dots+\theta_nx_n(t)}=n,
\end{equation*} and $A\subseteq \bigcup_{i=1}^n\supp(x_i)$. Observe that the former condition means that for every $t\in A$, all the non-zero $\theta_ix_i(t)$ have the same sign. We claim that there exist $x_{n+1}\in V_{n+1}$ and $s\in T$ such that $\abs{x_{n+1}(s)}=1$ and $A\cup\{s\}\in\A_f$. Indeed, if an element $y\in V_{n+1}$ would be such that $y(s)=0$ for every such $s$, then we could simply pick an antichain $\alpha:=(\alpha(n))\succ A$ and find $n_0\in\N$ large enough so that $y+e_{\alpha(n_0)}$ would belong to $V_{n+1}$. Then $x_{n+1}=y+e_{\alpha(n_0)}$ would satisfy the required condition with $s=\alpha(n_0)$. Now given such $x_{n+1}$ and $s$, we are left with two cases to consider. If $s\notin A$, then, taking $\theta_{n+1}=1$, we have
\begin{equation*}
    \norm{\theta_1x_1+\dots+\theta_nx_n+\theta_{n+1}x_{n+1}} \geq \sum_{t\in A\cup\{s\}}\abs{\theta_1x_1(t)+\dots+\theta_nx_n(t)+x_{n+1}(t)}=n+1,
\end{equation*} so we are done. Otherwise, let $\theta_{n+1}\in\{-1,1\}$ be such that $\theta_{n+1}x_{n+1}(s)$ has the same sign as all the non-zero $\theta_ix_i(s)$. Then \begin{equation*}
    \abs{\theta_1x_1(s)+\dots+\theta_nx_n(s)+\theta_{n+1}x_{n+1}(s)}=\abs{\theta_1x_1(s)+\dots+\theta_nx_n(s)}+1,
\end{equation*} and it immediately follows that $\norm{\theta_1x_1+\dots+\theta_nx_n+\theta_{n+1}x_{n+1}}=n+1$, as we wanted.
\end{proof}

\subsection{A renorming of the binary tree space whose unit ball is not SCD}\label{subsec:binary_tree_renorming} Our aim here is to prove that the set $C:=\cconv(B_X^+\cup B_X^-)$ is also not SCD. As a corollary, we will get an example of a Banach space with a 2-unconditional basis whose unit ball is not SCD.

First, let us observe that the set $\Pi^+$ is not generating for $B_X^+$. Indeed, it is easy to show that if $\Omega$ denotes the set $\Sigma\setminus \Pi$, then no point in $\Omega^+$ belongs to the closed convex hull of $\Pi^+$. However, the set $\Sigma^+=\Omega^+\cup\Pi^+$ is generating for $B_X^+$. 

\begin{lemma}
    The set $\Sigma^+$ is generating for $B_X^+$.
\end{lemma}

\begin{proof}
    For every $n\in\N$, let $\Sigma_n:=\Sigma\cap X_n$. Then it is straightforward to check that $\Sigma_n^+=\ext(B_{X_n}^+)$. So the conclusion follows using the exact same argument as in Subsection~\ref{subsec:binary_tree_dentability}.
\end{proof}

Let us recall the concept of slicely countably determined points, which was introduced in \cite{LLMRZ24} as a localization or point-wise version of the global SCD property to study some SCD phenomena in the non-separable setting. Given a bounded convex subset $A\subseteq X$, we say that $x\in A$ is a \emph{slicely countably determined point of $A$} (SCD point of $A$ for short, writing $x\in\SCD(A)$), if there is a countable collection of slices $\{S_n\}_{n\in\mathbb{N}}$ of $A$ such that every sequence $(x_n)\in \prod_{n\in\N} S_n$ is generating for the point $a$, meaning that $a\in\cconv\{x_n\colon n\in\mathbb{N}\}$. We refer to \cite{LLMRZ24} for background on SCD points. In particular, it was proved there that the set $\SCD(A)$ is always closed and convex, and that if $A$ is additionally assumed to be separable, then $A$ is SCD if and only if every element of $A$ is an SCD point of $A$. In Theorem~\ref{thm:binary_tree_space_is_not_SCD}, we have essentially proved that 0 is not an SCD point of the positive unit ball of $X$. We can actually produce more of these points.   

\begin{proposition}\label{prop: no x in omega+ is SCD of B_X+}
     No element $x\in\Omega^+$ is an SCD point of $B_{X^+}$. 
\end{proposition}

\begin{proof}
    Let $x\in\Omega^+$ and let $(S_n)$ be a sequence of slices of $B_X^+$. Then we can find $A\in\A_f$ for which the set $S:=A\cup T(\max(A))$ does not intersect the support of $x$. Let $P_S$ be the norm one projection from $X$ onto $[e_s]_{s\in S}$. Using the exact same argument as the one presented in the proof of Theorem~\ref{thm:binary_tree_space_is_not_SCD}, we can construct inductively a sequence $(x_n)\in \prod_{n\in\N} S_n$ of finitely supported elements satisfying \begin{equation*}
        \norm{P_S(x_1+\dots+x_n)}=n
    \end{equation*} for every $n\in \N$. In particular, there exists $f\in S_{X^*}$ supported on $S$ (meaning that $f(e_t)=0$ for every $t\in T\setminus S$) such that $f(x_n)=1$ for every $n\in \N$. Since $S\cap\supp(x)=\emptyset$, we have $f(x)=0$, hence $x\notin \cconv\{x_n\colon n\in\N\}$. The conclusion follows. 
\end{proof}

\begin{proposition}
 Let $x\in S_X^+$. If $x$ is an SCD point of $C$, then $x$ is an SCD point of $B_{X^+}$. In particular, $C$ is not an SCD set. 
\end{proposition}

\begin{proof}
   Let $(S_n)$ be a sequence of slices of $C$ which is determining for some element $x\in S_{X^+}$, and let $\M$ be the set of all $n\in\N$ for which $S_n$ intersects $B_X^+$. Then the sequence $(S_n\cap B_X^+)_{n\in\M}$ is also determining for $x$. Indeed, let $(x_n)_{n\in \M}\in\prod_{n\in\M}S_n\cap B_X^+$, and fix $\eps\in(0,1)$. For every $n\in \N\setminus\M$, pick an arbitrary $x_n\in S_n$. Since $(S_n)$ is a determining sequence for $x$ in $C$, we can find $y\in \conv\{x_n\colon n\in\M\}$, $z\in \conv\{x_n\colon n\notin \M\}$ and $\lambda\in[0,1]$ such that \begin{equation*}
   \norm{x-(\lambda y+(1-\lambda)z)}<\eps.
   \end{equation*} Let $\beta$ be a branch of $T$ such that $\sum_{t\in\beta}x(t)=1$. Since $z\in B_X^-$, we have \begin{equation*}
      1-\lambda\leq 1-\lambda\sum_{t\in\beta}y(t)-(1-\lambda)\sum_{t\in\beta}z(t)\leq\norm{x-(\lambda y+(1-\lambda)z)}<\eps.
   \end{equation*}
   Therefore, \begin{equation*}
       \norm{x-y}\leq \norm{x-(\lambda y+(1-\lambda)z)}+2(1-\lambda)<3\eps,
   \end{equation*} which means that $x\in\cconv\{x_n\colon n\in\M\}$, as we wanted.
   
   The last part of the claim follows from Proposition \ref{prop: no x in omega+ is SCD of B_X+}, yielding that no $x\in\Omega^+$ is an SCD point of $C$. By \cite[Lemma~2.8]{LLMRZ24} $C$ cannot be an SCD set.
\end{proof}

As a corollary, we the following. 

\begin{theorem}
    There exists a Banach space $Y$ with a 2-unconditional basis whose unit ball is not SCD.
\end{theorem}

\begin{proof}
    Let $Y:=(X,\nnorm{\cdot})$, where $\nnorm{\cdot}$ is the equivalent norm on $X$ given by the gauge (a.k.a. Minkowski functional) of the set $C$. Since $\frac{1}{2}B_X\subseteq C \subseteq B_X$, we have that $\norm{\cdot}\leq \nnorm{\cdot}\leq 2\norm{\cdot}$. In particular, the family $(e_t)_{t\in T}$ is 2-unconditional in $Y$. Now by the above, $B_Y=C$ is not an SCD set, so we are done.

\end{proof}

\subsection{The unit ball of the binary tree space has a countable $\pi$-base for the relative weak topology}\label{subsec:weak_density_of_points_of_continuity} In this subsection, we show that the points of continuity of $B_X$ are weakly dense in $B_X$. In particular, $(B_X,w)$ has countable $\pi$-base.

\begin{proposition}\label{prop:points_of_continuity_characterization}
    A point $x\in S_X$ is a point of continuity of $B_X$ if and only if $\norm{P_{\beta}(x)}=1$ for every branch $\beta$ of $T$.
\end{proposition}

\begin{proof}
First, assume that there exists $\eps>0$ and a branch $\beta$ of $T$ such that $\norm{P_{\beta}(x)}<1-\varepsilon$. Since $x$ belongs to the closure of $c_{00}(T)$, we can find $s\in\beta$ such that $\norm{P_{T(s)}(x)}<\frac{\eps}{2}$. Let $\alpha:=(\alpha(n))\succ s$ be an infinite antichain of $T(s)$, and for every $n\in\N$, define $x_n:=x+\frac{\eps}{2}e_{\alpha(n)}$. By construction, $x_n\in B_X$ for every $n\in \N$, the sequence $(x_n)$ converges to $x$ weakly, and $\norm{x-x_n}\geq \frac{\eps}{2}$ for every $n\in\mathbb{N}$, so $x$ is not a point of continuity of $B_X$. 

Second, assume that $\norm{P_{\beta}(x)}=1$ for every branch $\beta$ of $T$, pick a net $(x_a)$ in $B_X$ which converges weakly to $x$, and fix $\eps>0$. Then take $n\in\N$ such that $\norm{P_{\perp_n}(x)}<\eps$. Note that because of our assumption, this means that $\sum_{t\preceq s}\abs{x(t)}>1-\eps$ for every $s\in L_n$. Since $(x_a)$ converges weakly to $x$, we can find $a$ in the corresponding directed set such that for every $b \geq a$ and for every $s\in T_n$, we have $\abs{x_b(s)-x(s)}<\frac{\eps}{n+1}$. In particular, $\norm{P_{T_n}(x_b-x)}<\eps$. Furthermore, for every $s\in L_n$, we have \begin{equation*}
    \sum_{t\preceq s}\abs{x_b(t)}>\sum_{t\preceq s}\abs{x(t)}-\eps>1-2\eps.
\end{equation*} Therefore, $\norm{P_{\perp_n}(x_b)}<2\eps$, and it immediately follows that $\norm{x_b-x}<4\eps$. So $(x_a)$ converges in norm to $x$, and $x$ is a point of continuity of $B_X$.  
\end{proof}

Let $\B$ be the set of all branches of $T$. For every $\beta\in\B$ and every $f\in X^*$, we set \begin{equation*}
    l_\beta(f):=\limsup_{t\in\beta}\abs{f(e_t)}.
\end{equation*} The following lemma is straightforward.

\begin{lemma}
    For every $f\in X^*$, the series $\sum_{\beta\in\B}l_\beta(f)$ converges. In particular, $l_\beta(f)=0$ for all but countably many $\beta$ in $\B$.
\end{lemma}

\begin{proof}
 Let $(\beta_j)_{j=1}^k$ be a finite sequence in $\B$. Then we can find $N\in\N$ such that every sequence $(s_j)_{j=1}^k$ in $\prod_{j=1}^k(\beta_j\cap \perp_N)$ is an antichain of $T$. In particular, for every such sequence, we have \begin{equation*}
 \sum_{j=1}^k\abs{f(e_{s_j})}=\scal{f,\sum_{j=1}^k\sign(f(e_{s_j}))e_{s_j}}\leq \norm{f}.
 \end{equation*} Therefore, \begin{equation*}
     \sum_{j=1}^kl_{\beta_j}(f)\leq \norm{f},
 \end{equation*} and conclusion follows. 
\end{proof} From this, we deduce that every functional $f\in X^*$ has small values outside of a set of the form $T_n\cup \bigcup_{j=1}^k\beta_j$ in the following sense. 

\begin{lemma}\label{lem:functionals_have_small_values_on_a_big_part_of_the_tree}
Let $f\in X^*$ and $\varepsilon>0$. Then there exist $\beta_1,\dots,\beta_k\in \B$ and $n\in\N$ such that for every $s\in \perp_n \setminus \bigcup_{i=1}^k \beta_i$, we have $|f(e_s)|<\varepsilon$.
\end{lemma}

\begin{proof}
By the previous lemma, there exists $\beta_1,\dots,\beta_k\in \B$ such that $l_\beta(f)<\varepsilon$ for every $\beta\in\B\setminus\{\beta_1,\dots,\beta_k\}$. Suppose that for every $n\in\N$, we can find $s\in \perp_n \setminus \bigcup_{j=1}^k \beta_j$ such that $\abs{f(e_s)}\geq\varepsilon$. Then we can construct inductively a sequence $(s_i)$ in $T\setminus \bigcup_{j=1}^k \beta_j$ such that $\abs{s_{i+1}}>\abs{s_i}$ and $\abs{f(e_{s_i})}\geq\varepsilon$ for every $i$. Now observe that by the latter condition, there can be at most $\frac{\norm{f}}{\eps}$ of these points which are incomparable for $\preceq$. Therefore, we can find a branch $\beta$ of $T$ which contains infinitely many of the $s_i$. By construction, $\beta\in\B\setminus\{\beta_1,\dots,\beta_n\}$, and $l_\beta(f)\geq \eps$, contradicting our initial claim. 
\end{proof}

Given $n\in\N$ and $a_1,\dots,a_n$ in the interval $(-\eps,\eps)$, we can always choose inductively $\theta_1,\dots,\theta_n\in\{-1,1\}$ such that $\sum_{i=1}^m\theta_ia_i$ belongs to $(-\eps,\eps)$ for every $m\leq n$. The following lemma allows us, given $k$ such sequences of points $a_1^j,\dots,a_n^j\in(-\eps,\eps)$, $j\in\{1,\dots,k\}$, to select a common sequence of signs $\theta_1,\dots,\theta_n\in\{-1,1\}$ for which all the sums $\sum_{i=1}^n\theta_ia_i^j$ will belong to the interval $(-2^k\eps,2^k\eps)$.
\begin{lemma}\label{lem:Jaani_lemma}
    Let $k\in\N$ and $n>2$. For every $j\in\{1,\dots,k\}$, let $a^j:=(a_i^j)_{i=1}^n\in B_{\ell_\infty^n}$. Then there exists $\theta:=(\theta_i)_{i=1}^n\in\{-1,1\}^n$ such that for every $j\in\{1,\dots,k\}$, we have \begin{equation*}
        \abs{\sum_{i=1}^n\theta_ia^j_i}\leq 2^k.
    \end{equation*}
\end{lemma}

\begin{proof}
        Fix $k\in\N$. We will prove the statement by induction on $n$. First, notice that if $n\leq2^k$, then every sequence of signs gives us what we want simply by the triangle inequality. Next, assume that the statement holds for $n\geq 2^k$, and fix $k$ vectors $a^1,\dots, a^k$ in $B_{\ell_{\infty}^{n+1}}$. Notice that for every $i\in\{1,\dots, n+1\}$, the sequence $(\sign(a_i^j))_{j=1}^k$ defines an element of $\{-1,1\}^k$ (setting $\sign(0):=1$). Since by assumption $n+1>2^k=\abs{\{-1,1\}^k}$, there are at least two such indices which define the same sequence. In other words, we can find $i_1<i_2\in\{1,\dots,n+1\}$ such that $\sign(a^j_{i_1})=\sign(a^j_{i_2})$ for every $j\in\{1,\dots,k\}$. Set $\omega_1:=1$ and $\omega_2:=-1$. Then we have $|\omega_1a^j_{i_1}+\omega_2a^j_{i_2}|\leq 1$ for every $j\in\{1,\dots,k\}$. Now define vectors $b^j=(b^j_i)_{i=1}^{n}\in\ell_{\infty}^n$ by setting 
       \begin{equation*}
        b^j =\big(\omega_1a^j_{i_1}+\omega_2a^j_{i_2},a^j_1,\dots,a^j_{i_1-1},a^j_{i_1+1},\dots,a^j_{i_2-1},a^j_{i_2+1},\dots,a^j_{n+1}\big),\;\; j\in\{1,\dots,k\}.
        \end{equation*}
        By construction, we have $||b^{j}||_{\infty}\leq 1$ for every $j$. So by assumption, there exists $\rho_1,\dots,\rho_n\in\{-1,1\}$ such that for every $j$, we have \begin{equation*}
        \abs{\sum_{i=1}^n\rho_ib^j_i}\leq 2^k.
        \end{equation*}
       Let $\theta=(\rho_2,\dots,\rho_{i_1},\rho_1\omega_1,\rho_{i_1+1},\dots,\rho_{i_2-1},\rho_1\omega_2,\rho_{i_2},\dots,\rho_n)\in\{-1,1\}^{n+1}$. Then, for every $j$, we have \begin{equation*}  
        \abs{\sum_{i=1}^{n+1}\theta_i a^j_i}=\abs{\sum_{i=1}^n\rho_ib^j_i}\leq 2^k.
       \end{equation*} The conclusion follows.
    \end{proof} 

With these tools at hand, we can now prove the main result of this subsection. First, we show that $0$ belongs to the weak closure of the set $\PC(B_X)$. 
    
\begin{proposition}\label{prop:prop:0_in_weak_closure_of_set_of_points_of_continuity}
 Every neighborhood of $0$ in the realtive weak topology of $B_X$ contains a point of continuity of $B_X$.
\end{proposition}

\begin{proof}
    Let $f_1,\dots,f_k\in S_{X^*}$, $\varepsilon>0$, and
    \begin{equation*}
        W =\bigcap_{j=1}^k\{ y\in B_X\colon |f_j(y)|<\varepsilon\}.
    \end{equation*}
By Lemma~\ref{lem:functionals_have_small_values_on_a_big_part_of_the_tree}, we can find branches $\beta_1,\dots,\beta_l\in\B$ and $N\in\N$ so that for every $s\in\perp_{N}\setminus\cup_{i=1}^l\beta_i$ and every $j=\{1,\dots,k\}$, we have $|f_j(e_s)|<\frac{\eps}{2^{k+1}}$. Wlog, the points $\beta_1(N),\dots,\beta_l(N)$ are incomparable for $\preceq$. Furthermore, by compactness, we can find $n_2>n_1>N$ in $\N$ such that for every $i,j$, we have \begin{equation*}
    \abs{f_j(e_{\beta_i(n_2)})-f_j(e_{\beta_i(n_1)})}<\varepsilon/l.
\end{equation*} For every $i\in\{1,\dots,l\}$, let \begin{equation*}
    S_i:=L_{n_2}\cap T(\beta_i(n_1))\setminus\{\beta_i(n_2)\}=\{s\in L_{n_2}\setminus\{\beta_i(n_2)\}\colon s\succeq \beta_i(n_1)\} 
\end{equation*} and $S:=\bigcup_{i=1}^lS_i$. Also, let \begin{equation*}
    R:=L_{n_2}\setminus\bigcup_{i=1}^lT(\beta_i(n_1)).
\end{equation*} By assumption, for every $t\in S\cup R$ and $j\in\{1,\dots,k\}$, we have $\abs{f_j(e_t)}<\frac{\eps}{2^{k+1}}$. Therefore, by Lemma~\ref{lem:Jaani_lemma}, we can find $(\theta_t)_{t\in S\cup R}\subseteq\{-1,1\}$ such that for every $j\in\{1,\dots,k\}$, we have \begin{equation*}
    \abs{\sum_{t\in S}\frac{1}{2}\theta_tf_j(e_t)+\sum_{t\in R}\theta_t f_j(e_t)}<\frac{\eps}{2}.
\end{equation*} So consider the point $x$ in $X$ given by \begin{equation*}
    x:=\sum_{i=1}^l\frac{1}{2}(e_{\beta_i(n_2)}-e_{\beta_i(n_1)})+\sum_{t\in S}\frac{1}{2}\theta_te_t+\sum_{t\in R}\theta_t e_t.
\end{equation*} By construction, $x\in B_X$, and for every branch $\beta$ of $T$, we have $\norm{P_\beta(x)}=1$, so it follows from Proposition~\ref{prop:points_of_continuity_characterization} that $x$ is a point of continuity of $B_X$. Furthermore, by the above, we have that for every $j\in \{1,\dots,k\}$, \begin{equation*}
    \abs{f_j(x)}\leq\frac{1}{2}\sum_{i=1}^l\abs{f_j(e_{\beta_i(n_2)})-f_j(e_{\beta_i(n_1)})}+\abs{\sum_{t\in S}\frac{1}{2}\theta_tf_j(e_t)+\sum_{t\in R}\theta_t f_j(e_t)}<\eps.
\end{equation*} So $x\in W$, as we wanted. 
\end{proof}

For every $t\in T$, let $S_t$ be the linear isometry from $X$ onto $[e_s]_{s\in T(t)}$ defined by the formula \begin{equation*}
    S_t(e_s):=e_{t\smallfrown s}
\end{equation*} for every $s\in T$, and extended by linearity and density to the whole space $X$. Using these shift operators and the previous result, we can now prove the following.

\begin{theorem}\label{thm:weak_density_PC(B_X)}
    The set $PC(B_X)$ is weakly dense in $B_X$. In particular, $(B_X,w)$ has a countable $\pi$-base.
\end{theorem}

\begin{proof}
 By Proposition~\ref{prop:prop:0_in_weak_closure_of_set_of_points_of_continuity}, there exists a net $(x_a)$ in $\PC(B_X)$ which converges weakly to 0. So let $y\in B_X$ be a non-zero finitely supported element and let $n\in\N$ be such that such that $\supp(y)\subseteq T_{n-1}$. For every $t\in L_n$, let \begin{equation*}
     \lambda_t:=\sum_{s\preceq t}\abs{y(s)},
 \end{equation*} and define \begin{equation*}
     y_a:=y+\sum_{t\in L_n}(1-\lambda_t)S_t(x_a).
 \end{equation*} for every $a$. By construction, $y_a\in B_X$, and $\norm{P_\beta(y_a)}=\lambda_{\beta(n)}+1-\lambda_{\beta(n)}=1$ for every branch $\beta$ of $T$, so $y_a\in \PC(B_X)$ by Proposition~\ref{prop:points_of_continuity_characterization}. Since the $S_t$ are bounded linear operators, the net $(y_a)$ converges weakly to $y$, so we are done. 
\end{proof}

\begin{remark}
    In \cite[Section~4]{ALMT}, the authors produced another adequate family $\mathfrak{M}$ on the unrooted binary tree $T^*=T\setminus\{\emptyset\}$ by adding to $\A^*:=\{A\cap T^*\colon A\in\A\}$ all the $\lambda$-segments and their subsets, i.e. subsets of $T^*$ of the form $\{t\in A\colon t\preceq s\}\cup\{s\smallfrown 0,s\smallfrown 1\}$ where $s$ is an arbitrary point of $T$ and $A\in\A^*$. By this process, they not only removed all the strongly exposed points from the unit ball of $X$, but actually obtained a Banach space $X_\mathfrak{M}$ with some remarkable geometric properties: this space has a 1-unconditional basis, the set $D$ of all Daugavet points of $X_\mathfrak{M}$ is weakly dense in $B_{X_\mathfrak{M}}$ (in particular, $X_\mathfrak{M}$ has the slice diameter 2 property), and the set $D_B:=D\cap \PC(B_{X_\mathfrak{M}})$ is generating for $B_{X_\mathfrak{M}}$ (in particular, $X_\mathfrak{M}$ fails the diameter 2 property). We would like to point out that the proof we presented in this section can easily be adapted to this setting to show that actually the set $D_B$ is also weakly dense in $B_{X_\mathfrak{M}}$ (and consequently, $(B_{X_\mathfrak{M}},w)$ also has a countable $\pi$-base). Indeed, in order to prove that $\PC(B_{X_\mathfrak{M}})$ is weakly dense in $B_{X_\mathfrak{M}}$, then the only part in the above scheme that requires non-trivial adjustments is the proof of Proposition~\ref{prop:prop:0_in_weak_closure_of_set_of_points_of_continuity}, because the elements we construct there are not in the unit ball of $X_\mathfrak{M}$. This can be fixed by initially asking \begin{equation*}
        \abs{f_j(e_s)}<\frac{\eps}{2^{2k+1}}
    \end{equation*} for every $s\in\perp_{N}\setminus\cup_{i=1}^l\beta_i$ and every $j=\{1,\dots,k\}$, and then by replacing for every $t\in S\cup R$ the point $e_t$ which appears in the definition of the point $x$ there by the point  \begin{equation*}
        y_t:=\sum_{s\in T^*}\frac{1}{2^{\abs{s}}}\omega_{t,s}e_{t\smallfrown s}, 
    \end{equation*} where the $\omega_{t,s}$ are signs which are obtained level by level from Lemma~\ref{lem:Jaani_lemma} in such a way that \begin{equation*}
        \abs{f_j(y_t)}\leq \sum_{n\in\N}\frac{1}{2^n}\abs{\sum_{\abs{s}=n}\omega_{t,s}f_j(e_{t\smallfrown s})}\leq 2^k\frac{\eps}{2^{2k+1}}\sum_{n\in\N}2^{-n}=\frac{\eps}{2^{k+1}}
    \end{equation*} for every $j\in\{1,\dots,k\}$. Then we chose the sequence $(\theta_t)_{t\in S\cup R}$ by calling a second time to Lemma~\ref{lem:Jaani_lemma} to get \begin{equation*}
        \abs{\sum_{t\in S}\frac{1}{2}\theta_tf_j(y_t)+\sum_{t\in R}\theta_tf_j(y_t)}<2^k\frac{\eps}{2^{k+1}}=\frac{\eps}{2}. 
    \end{equation*} We leave the details to the reader. To conclude, it remains to see that actually $D_B=\PC(B_{X_\mathfrak{M}})$. This is easily seen by combining several of the observations already made in \cite{ALMT}, but let us point out that this actually also immediately follows from the weak density and norm closedness of the set $D$. Indeed, if a point $x\in \PC(B_X)$ belongs to the weak closure of some subset $A$ of $B_X$, then it actually belongs to the norm closure of $A$, so we get the desired result. 
\end{remark}

\section{The countably branching tree space}\label{section:countably_branching_tree_space}

It is clear that the results from the previous section generalize to any infinite finitely branching tree (i.e. a tree such that every element of it has only finite amount of immediate successors). However, the situation is different when one allows infinite branching, because one then loses the weak density of the points of continuity in the unit ball of the associated Banach space. In particular, let
 $T_{\infty}=\bigcup_{k\in\N}\N^k$ be the countably branching tree and $\A_\infty$ be the collection of all chains of $T_{\infty}$. In this case, the Banach space $X_{T_\infty}$ generated by the family $\A_\infty$ has the strong diameter 2 property, meaning that all convex combination of slices (hence, by Bourgain lemma, all non-empty relatively weakly open subsets) of $B_{X_{T_\infty}}$ have diameter 2. In particular, $B_{X_{T_\infty}}$ has no point of continuity. Nevertheless, we will prove in this section that $(B_{X_{T_\infty}},w)$ has a countable $\pi$-base. The key idea for this is the following lemma, which shows that every functional on $X_\infty$ is essentially supported on a finitely branching tree.

\begin{lemma}\label{lemma:functinals_on_X_infty_are_essentially_supported_on_a_finitely_branching_tree}
    Let $f\in X_{T_\infty}^*$ and $\varepsilon>0$. There exists a finitely branching tree $T_0\subseteq T_{\infty}$ such that $||P^{*}_{T_{\infty}\setminus T_0}f||<\varepsilon$.
\end{lemma}

\begin{proof}
 Fix a non-zero $f\in X_{T_\infty}^*$ and $\varepsilon>0$. For convenience, given two subsets $A,B$ of $T_\infty$, we will write $A\preceq B$ if every element of $B$ is a successor of some element of $A$.  First, we construct inductively a sequence $(R_k)_{k\in \N\cup\{0\}}$ of finite subsets of $T_\infty\cup\{\emptyset\}$ with $R_0=\{\emptyset\}$ and such that for every $k\in\mathbb{N}$, the following conditions are met.
    \begin{enumerate}
        \item $R_k\subseteq L_k$;
        \item $R_k\succeq R_{k-1}$;
        \item $||P^*_{S_k}f||<\varepsilon/2^k$, where $S_k = \bigcup_{s\in L_k\setminus R_k}T(s)$.
    \end{enumerate} Indeed, let $k=0$ or assume that for some $k\in\N$ we have found $R_k$ satisfying the conditions (1), (2) and (3). Observe that \begin{equation*}
        ||P^*_{\bigcup_{\substack{s\in L_{k+1}\\
    s\succeq R_k}}T(s)}f|| = \sum_{\substack{s\in L_{k+1}\\ s\succeq R_k}} ||P^*_{T(s)}f||,
    \end{equation*}
    so we can find a finite subset $R_{k+1}\succeq R_k$ of $L_{k+1}$ such that\begin{equation*}
        \sum_{s\in R_{k+1}}||P^*_{T(s)}f|| >  ||P^*_{\bigcup_{\substack{s\in L_{k+1}\\ s \succeq R_{k}}}T(s)}f||- \frac{\varepsilon}{2^{k+1}}.
    \end{equation*}
    In particular, \begin{equation*}
        ||P^*_{S_{k+1}}f|| = ||P^*_{\bigcup_{s\in L_{k+1}\setminus R_{k+1}}T(s)}f|| = \sum_{s\in L_{k+1}\setminus R_{k+1}}||P^*_{T(s)}f||<\frac{\varepsilon}{2^{k+1}},
    \end{equation*} so $R_{k+1}$ satisfies all the required conditions. Finally, let $T_0:= \bigcup_{k\in\mathbb{N}} R_k$. Then $T_0$ is a finitely branching subtree of $T_\infty$, and 
\begin{equation*}
    T_{\infty}\setminus T_0 = \bigcup_{k\in\mathbb{N}}\bigcup_{s\in L_k\setminus R_k} T(s) = \bigcup_{k\in\mathbb{N}}S_k.
\end{equation*} Since the $S_k$ are totally incomparable, it follows that \begin{equation*}
     || P^*_{T_{\infty}\setminus T_0}f|| = ||P^*_{\bigcup_{k\in\mathbb{N}}S_k}f|| = \sum_{k\in\mathbb{N}}||P^*_{S_k}f|| < \sum_{k\in\mathbb{N}}\frac{\varepsilon}{2^k}=\varepsilon,
\end{equation*} as we wanted.
\end{proof}

With this result at hand, we can prove the following. 

\begin{theorem}\label{thm:countably_branching_tree_has_a_countable_pi-base}
The set $(B_{X_{T_\infty}},w)$ has a countable $\pi$-base.
\end{theorem}

\begin{proof}
We will show that every non-empty relatively weakly open subset $W$ of $(B_{X_{T_\infty}},w)$ contains a subset of the form
\begin{equation*}
    W(x_0,\delta_0):=\bigcap_{s\in\supp(x_0)}\Big\{y\in B_{X_{T_\infty}}\colon \abs{x_0(s)-y(s)}<\delta_0\Big\},
\end{equation*}
where $x_0\in X_{T_\infty}$ is a finitely supported element and $\delta>0$.  Take $W$ of the form
\begin{equation*}
    W:=\bigcap_{i=1}^k\{ y\in B_X \colon |f_i(x-y)|<\varepsilon\},
\end{equation*}
where $x\in X_{T_\infty}$ is a finitely supported element, $f_1,\dots,f_k\in S_{X^*}$ and $\varepsilon>0$. By the previous lemma, we can find a finitely branching subtree $T_0$ of $T_\infty$ such that $||P^*_{T_{\infty}\setminus T_0}f_i||<\frac{\eps}{4}$ for every $i\in\{1,\dots,k\}$.  Wlog, $\supp(x)\subseteq T_0$. Let $X_0:=[e_s]_{s\in T_0}$ and 
\begin{equation*}
    W_0 =\bigcap_{i=1}^k\Big\{z\in B_{X_0}\colon |f_i(x-z)|<\frac{\varepsilon}{4}\Big\}.
\end{equation*}
As already mentioned, we can use the scheme from the previous section with the finitely branching tree $T_0$ to show that $W_0$ contains a finitely supported point of continuity $x_0$ of $B_{X_0}$. A quick look at the proof of Proposition~\ref{prop:points_of_continuity_characterization} will convince the reader that in this case, the diameter of the set $W(x_0,\delta)\cap B_{X_0}$ goes to 0 as $\delta$ goes to 0. Pick $\delta_0>0$ such that $\diam (W(x_0,\delta_0)\cap B_{X_0})<\varepsilon/4$ and pick $y\in W(x_0,\delta_0)$. Note that since $\supp(x_0)\subseteq T_0$, we have that $P_{T_0}y\in W(x_0,\delta_0),$ and thus $\norm{x_0-P_{T_0}y}<\frac{\eps}{4}$. Hence, for every $i\in\{1,\dots,k\}$, we have
\begin{align*}
    |f_i(x-y)| &\leq |f_i(P_{T_0}(x-y))| + |f_i(P_{T_{\infty}\setminus T_0}(x-y))| \\
    &\leq |f_i(x-x_0)|+|f_i(x_0-P_{T_0}y)|+2||P^*_{T_{\infty}\setminus T_0}f_i||\\
    & < \frac{\varepsilon}{4}+\frac{\varepsilon}{4}\norm{f_i}+2\frac{\varepsilon}{4} = \varepsilon.
\end{align*} Therefore $W(x_0,\delta_0)\subseteq W$. The conclusion follows.
\end{proof}

\section{Additional comments and open questions}\label{section:open_questions}
As previously mentioned, it is currently unknown whether or not every convex SCD set admits a countable $\pi$-base for the relative weak topology. In particular, the following also seems to be open.

\begin{question}
    Let $X$ be a Banach space with a 1-unconditional basis. Does $B_X$ admit a countable $\pi$-base in the relative weak topology?
\end{question}

One natural candidate in regard to this latter question is the stopping time Banach space that we already mentioned in the introduction, and which can be seen as the Banach space generated by the family of all antichains in the binary tree (see \cite{BO89} for more details). This space also has a 1-unconditional basis, but the reader can easily check that its unit ball contains no points of continuity. We do not know whether or not the unit ball of this space admits a countable $\pi$-base for the relative weak topology.

One of the main difficulties when facing this sort of problems is that there are very few tools at our disposal to show the non-existence of a $\pi$-base. In Section~\ref{section:preliminaries}, we have seen that the super alternative Daugavet property naturally prevents the existence of a countable $\pi$-base for the relative weak topology of the unit ball of the underlying Banach space. So far, there are few known examples of Banach spaces which satisfy this property while failing the Daugavet property. In particular, the following is currently unknown (see \cite[Section~6]{LLMRZP25} for more details), and would be another very natural approach to the above question.

\begin{question}
    Does there exist a Banach space $X$ with the super alternative Daugavet property whose unit ball is an SCD set? In particular, does there exist a Banach space with a 1-unconditional basis with the super alternative Daugavet property?
\end{question}

A direct consequence of our results is that both the binary tree and the countably branching tree spaces fail to have the super alternative Daugavet property. We would like to point out that this is actually also the case for every Banach space $X$ generated by an adequate family. Recall that a collection $\A$ of subsets of $\N$ is called an \emph{adequate family} if \begin{enumerate}
    \item $\A$ contains all the singletons: for every $n\in \N$, $\{n\}\in\A$;
    \item $\A$ is hereditary: if $A\in\A$ and $B\subseteq A$, then $B\in\A$;
    \item $\A$ is compact with respect to the topology of pointwise convergence: for every $A\subseteq\N$, if $B\in\A$ for every finite $B\subseteq A$, then $A\in\A$. 
\end{enumerate} Given such a family, we consider the Banach space $h_{\A,1}$ obtained as the completion of the space $(c_{00},\norm{\cdot})$, where $\norm{\cdot}$ is the norm on $c_{00}$ given by \begin{equation*}
    \norm{\sum_{n\in\N}a_ne_n}=\sup_{A\in\A} \sum_{n\in A}\abs{a_n}
\end{equation*}for every $(a_n)\in c_{00}$. Then the canonical basis $(e_n)$ of $c_{00}$ form a 1-unconditional basis for $h_{\A,1}$. Spaces generated by adequate families include e.g. $c_0$, $\ell_1$, $c_0(\ell_1)$ or the Schreier space, as well as the tree-spaces considered in the previous two sections. We refer to \cite{ABC21,ALM} and references therein for more information on these spaces. With respect to the latter question, in this setting, we can provide a negative answer. 

\begin{proposition}
 Let $\A$ be an adequate family of subsets of $\N$. Then the space $h_{\A,1}$ fails the super Alternative Daugavet property.
\end{proposition}

\begin{proof}
    Let $X:=h_{\A,1}$. If $\A$ contains only the singletons, then $X$ is isometrically isomorphic to the space $c_0$, and therefore fails to have the super ADP (this follows e.g. from \cite[Theorem~3.10]{LLMRZP25}). Otherwise, note that we can find $m<n$ in $\N$ such that $\{m,n\}\in\A$. We will show that the point $x:=\frac{1}{2}(e_m+e_n)$ is not a super alternative Daugavet point in the sense of \cite{LLMRZP25}. Indeed, let $y:=\frac{1}{2}(e_m-e_n)$ and let $(y_\alpha)$ be a net in $B_X$ which converges weakly to $y$. Then fix $\eps>0$ to be chosen later. By assumption, we can find $\alpha_0$ such that for every $\alpha\succeq\alpha_0$, we have \begin{equation*}
        \abs{y_\alpha(m)-\frac{1}{2}}<\eps\quad\text{and}\quad\abs{y_\alpha(n)+\frac{1}{2}}<\eps.
    \end{equation*} Take $\alpha\succeq \alpha_0$, $\theta\in\T$ and $A\in\A$. There are four cases to consider.

    \textbf{Case~1.} Assume that $m,n\in A$. On one hand, since $A\in\A$, we have \begin{equation*}
        \sum_{k\in A}\abs{y_\alpha(k)}\leq \norm{y_\alpha}\leq 1.
    \end{equation*} On the other, since $\alpha\succeq\alpha_0$, we have \begin{equation*}
        \sum_{k\in A}\abs{y_\alpha(k)}>1-2\eps+\sum_{k\in A\setminus\{m,n\}}\abs{y_\alpha(k)}.
    \end{equation*} Hence, \begin{equation*}
        \sum_{k\in A\setminus\{m,n\}}\abs{y_\alpha(k)}<2\eps,
    \end{equation*} and it follows that \begin{equation*}
        \sum_{k\in A}\abs{x(k)+\theta y_\alpha(k)}<\frac{1}{2}(\abs{1+\theta}+\abs{1-\theta})+4\eps. 
    \end{equation*}

    \textbf{Case~2 and 3.} Assume that $m\in A$ and $n\notin A$. Again, \begin{equation*}
        \sum_{k\in A}\abs{y_\alpha(k)}>\frac{1}{2}-\eps+\sum_{k\in A\setminus\{m\}}\abs{y_\alpha(k)},
    \end{equation*} and thus \begin{equation*}
        \sum_{k\in A\setminus\{m\}}\abs{y_\alpha(k)}<\frac{1}{2}+\eps.
    \end{equation*} It follows that \begin{equation*}
        \sum_{k\in A}\abs{x(k)+\theta y_\alpha(k)}<\frac{1}{2}\abs{1+\theta}+\frac{1}{2}+2\eps\leq \frac{3}{2}+2\eps. 
    \end{equation*} The case $n\in A$ and $m\notin A$ is analogous.

    \textbf{Case~4.} Assume that $m,n\notin A$. Then \begin{equation*}
        \sum_{k\in A}\abs{x(k)+\theta y_\alpha(k)}=\sum_{k\in A}\abs{y_\alpha(k)}\leq \norm{y_\alpha}\leq 1.
    \end{equation*}

    Summing up, we have shown that \begin{equation*}
        \norm{x+\theta y_\alpha}\leq \max\{\frac{3}{2}+2\eps,\frac{1}{2}(\abs{1+\theta}+\abs{1-\theta})+4\eps\}\leq \max\{\frac{3}{2}+2\eps,\sqrt{2}+4\eps\}.
    \end{equation*} So if $\eps$ was initially chosen small enough, we would get that this quantity is strictly smaller than 2, and therefore $x$ is not a super alternative Daugavet point, as claimed. 
\end{proof}

In Subsection~\ref{subsec:binary_tree_renorming}, we produced an example of a Banach space with a 2-unconditional basis $Y$ whose unit ball fails to be SCD. To the best of our knowledge, the following is unknown. 

\begin{question}
    Can we produce a Banach space with a $k$-unconditional basis for $k\in(1,2)$ whose unit ball is not SCD?
\end{question}

Also, observe that the unit ball of the latter space $Y$ does contain many SCD points (since $\cconv(\Pi^+\cup\Pi^-)\subseteq \SCD(B_Y)$). So it is natural to ask whether it is possible to do a little bit better. 

\begin{question}
    Does there exist a Banach space with an unconditional basis whose unit ball has no SCD point (equivalently, such that 0 is not an SCD point in this set)? 
\end{question}

The alert reader might wonder what would happen if we were simply to remove the strongly exposed points from the above set, or in other words to consider the symmetrization $D$ of the set $\Omega^+=\Sigma^+\setminus \Pi^+$ instead of $C$. It turns out that 0 is still SCD in this new unit ball. 

\begin{proposition}
    Let $D:=\cconv(\Omega^+\cup\Omega^-)\subseteq B_{X_T}$. Then $0\in\SCD(D)$. 

\end{proposition}

\begin{proof}
   For every $n\in\N$, let $x^{(n)}:=\frac{2^n-1}{2^n}\sum_{t\in L_n}e_t$. Then, for every $s\in L_n$, consider the norm-one functional $g_s:=\frac{1}{2^{n-1}}\sum_{t\in L_n\setminus\{s\}}e_t^*$. It is straightforward to check that for every $(x_{s,k}^{(n)})\in\prod_{s\in L_n}\prod_{k\in\N}S(D,g_s,\frac{1}{k})$, we have \begin{equation*}
        \limsup_{k\to \infty}\norm{x^{(n)}-\frac{1}{2^n}\sum_{s\in L_n}x_{s,k}^{(n)}}\leq \frac{1}{2^n}.
    \end{equation*} Hence, for every $(x_{s,k,\theta}^{(n)})\in \prod_{s\in L_n}\prod_{k\in\N}\prod_{\theta\in\{-1,1\}}S(D,\theta g_s,\frac{1}{k})$, we have \begin{equation*}
        \limsup_{k\to\infty} \norm{\frac{1}{2^{n+1}}\sum_{s\in L_n}x_{s,k,1}^{(n)}+\frac{1}{2^{n+1}}\sum_{s\in L_n}x_{s,k,-1}^{(n)}}\leq \frac{1}{2^{n-2}}.
    \end{equation*} Therefore, $0\in\cconv\{x_{s,k,\theta}^{(n)}\}$ for every $(x_{s,k,\theta}^{(n)})\in\prod_{n\in\N}\prod_{s\in L_n}\prod_{k\in\N}\prod_{\theta\in\{-1,1\}} S(D,\theta g_s,\frac{1}{k})$. The conclusion follows. 
\end{proof}

\section*{Acknowledgements}
The authors would like to thank Zdeněk Silber for introducing the paper by Bang and Odell \cite{BO89}, and Jaan Kristjan Kaasik for the proof of Lemma \ref{lem:Jaani_lemma}.

This research was supported by the Estonian Research Council grant (PRG2545).

\end{document}